\documentclass{amsart}
\usepackage[margin=0.8in]{geometry}

\usepackage{amsfonts}

\usepackage{setspace}
\usepackage{graphicx}
\usepackage{eufrak}
\usepackage{mathrsfs}
\usepackage{yfonts}
\usepackage{hyperref}
\usepackage{graphicx}
\usepackage{amsmath, amsthm, amscd, amsfonts, amssymb, graphicx, color}
\usepackage{url}
\usepackage{enumerate}
\makeatletter
\let\reftagform@=\tagform@
\def\tagform@#1{\maketag@@@{(\ignorespaces\textcolor{blue}{#1}\unskip\@@italiccorr)}}
\renewcommand{\eqref}[1]{\textup{\reftagform@{\ref{#1}}}}
\makeatother
\usepackage{hyperref}
\hypersetup{colorlinks=true, linkcolor=red, anchorcolor=green,
    citecolor=cyan, urlcolor=red, filecolor=magenta, pdftoolbar=true}

\usepackage{epsfig}        
\usepackage{epic,eepic}       

\newtheorem{theorem}{Theorem}
\theoremstyle{plain}

\newtheorem{corollary}{Corollary}

\newtheorem{lemma}{Lemma}

\newtheorem{proposition}{Proposition}
\newtheorem{remark}{Remark}

\numberwithin{equation}{section}

 \DeclareMathOperator{\spe}{sp}

\def\etal{et al.\,}

\begin{document}
\title[Operator   Popoviciu's inequality]
{Operator Popoviciu's inequality for superquadratic and convex functions of selfadjoint operators in Hilbert spaces}
\author[M.W. Alomari]{M.W. Alomari}

\address{Department of Mathematics, Faculty of Science and Information
	Technology, Irbid National University, 2600 Irbid 21110, Jordan.}
\email{mwomath@gmail.com}

\date{\today}

\begin{abstract}
In this work,   operator version of Popoviciu's inequality for
positive selfadjoint operators in Hilbert spaces under positive
linear maps for superquadratic  functions is proved. Analogously, using the same technique   operator version of Popoviciu's inequality   for convex functions is
obtained. Some other related inequalities are also deduced.
\end{abstract}

\maketitle
\section{Introduction}

Let $\mathcal{B}\left( \mathcal{H}\right) $ be the Banach algebra
of all bounded linear operators defined on a complex Hilbert space
$\left( \mathcal{H};\left\langle \cdot ,\cdot \right\rangle
\right)$  with the identity operator  $1_{\mathcal{H}}$ in
$\mathcal{B}\left(\mathcal{H} \right)$. Denotes
$\mathcal{B}^+\left( \mathcal{H}\right) $ the convex cone of all
positive operators on $\mathcal{H}$. A bounded linear operator $A$
defined on $\mathcal{H}$ is selfadjoint if and only if $
\left\langle {Ax,x} \right\rangle \in \mathbb{R}$ for all $x\in
\mathcal{H}$. For two selfadjoint operators  $A,B\in \mathcal{H}$,
we write $A\le B$ if $  \left\langle {Ax,x} \right\rangle \le
\left\langle {Bx,x} \right\rangle$ for all $x\in \mathcal{H}$.
Also, we define
\begin{align*}
\left\| A \right\| = \mathop {\sup }\limits_{\left\| x \right\| =
    1} \left| {\left\langle {Ax,x} \right\rangle } \right| = \mathop
{\sup }\limits_{\left\| x \right\| = \left\| y \right\| = 1}
\left| {\left\langle {Ax,y} \right\rangle } \right|.
\end{align*}

If $\varphi$ is any function defined on $\mathbb{R}$, we define
\begin{align*}
\left\| \varphi  \right\|_A  = \sup \left\{ {\left| {\varphi
        \left( \lambda  \right)} \right|:\lambda  \in \spe\left( A
    \right)} \right\}.
\end{align*}
If $\varphi$  is continuous then we write $\left\| \varphi
\right\|_A=\left\| A \right\|  $.

Let $A\in \mathcal{B}\left(\mathcal{H} \right) $ be a selfadjoint
linear operator on $\left( \mathcal{H};\left\langle \cdot ,\cdot
\right\rangle \right)$. Let $C\left(\spe\left(A\right)\right)$ be
the set of all continuous functions defined on the spectrum of $A$
$\left(\spe\left(A\right)\right)$ and let $C^*\left(A\right)$ be
the $C^*$-algebra generated by $A$ and the identity operator
$1_{\mathcal{H}}$.

Let us define the map $\mathcal{G}:
C\left(\spe\left(A\right)\right) \to C^*\left(A\right)$ with the
following properties (\cite{FMPS}, p.3):
\begin{enumerate}
    \item $\mathcal{G}\left(\alpha f + \beta g\right) = \alpha
    \mathcal{G}\left(f\right)+\beta \mathcal{G}\left(g\right)$, for
    all scalars $\alpha, \beta$.

    \item $\mathcal{G}\left(fg\right) = \mathcal{G}\left(f\right)
    \mathcal{G}\left(g\right)$ and
    $\mathcal{G}\left(\overline{f}\right)=\mathcal{G}\left(f\right)^*$;
    where $\overline{f}$ denotes to the conjugate of $f$ and
    $\mathcal{G}\left(f\right)^*$ denotes to the Hermitian of
    $\mathcal{G}\left(f\right)$.

    \item $\left\|\mathcal{G}\left(f\right)\right\|=\left\|f \right\|
    = \mathop {\sup }\limits_{t \in \spe\left(A\right)} \left|
    {f\left( t \right)} \right| $.

    \item $\mathcal{G}\left( {f_0 } \right) = 1_H$ and
    $\mathcal{G}\left( {f_1 } \right) = A$, where
    $f_0\left(t\right)=1$ and $f_1\left(t\right)=t$ for all $t \in
    \spe\left(A\right)$.
\end{enumerate}
Accordingly,  we define the continuous functional calculus for a
selfadjoint operator $A$ by
\begin{align*}
f\left(A\right) = \mathcal{G}\left(f\right)  \text{for all} \,f\in
C\left(\spe\left(A\right)\right).
\end{align*}
If both $f$ and $g$ are real valued functions on $\spe(A)$ then
the following important property holds:
\begin{align}
f\left( t \right) \ge g\left( t \right)  \,\text{for all} \, \,t
\in \spe\left( A \right) \,\,\text{implies}\,\, f\left( A \right)
\ge g\left( A \right), \label{eq1.1}
\end{align}
in the operator order of $\mathcal{B}\left(\mathcal{K} \right)$.

 A linear map  is defined to be   $\Phi:\mathcal{B}\left(\mathcal{H} \right)\to \mathcal{B}\left(\mathcal{K} \right)$ which preserves additivity and
homogeneity, i.e.,  $\Phi \left(\lambda_1 A +\lambda_2 B \right)=
\lambda_1\Phi \left( A  \right)+ \lambda_2\Phi \left( B \right)$
 for any $\lambda_1,\lambda_2 \in \mathbb{C}$  and $A, B \in \mathcal{B}\left(\mathcal{H} \right)$. The linear map is positive $\Phi:\mathcal{B}\left(\mathcal{H} \right)\to \mathcal{B}\left(\mathcal{K} \right)$ if it preserves the operator order, i.e., if $A\in \mathcal{B}^+\left(\mathcal{H} \right)$ then $\Phi\left(A\right)\in \mathcal{B}^+\left(\mathcal{K} \right)$, and in this case we write
 $\textfrak{B} [\mathcal{B}\left(\mathcal{H} \right),\mathcal{B}\left(\mathcal{K} \right)] $. Obviously, a positive linear map $\Phi$ preserves the order relation, namely
 $A\le B \Longrightarrow \Phi\left(A\right)\le \Phi\left(B\right)$ and preserves the adjoint operation $\Phi\left(A^*\right)=\Phi\left(A\right)^*$.
 Moreover, $\Phi$ is said to be  normalized (unital) if it preserves the identity operator, i.e. $\Phi\left(1_{\mathcal{H}}\right) = 1_{\mathcal{K}}$, in this case we write
$\textfrak{B}_n [\mathcal{B}\left(\mathcal{H}
\right),\mathcal{B}\left(\mathcal{K} \right)] $.

\subsection{Superquadratic functions}
A function $f:J\to \mathbb{R}$ is called convex iff
\begin{align}
f\left( {t\alpha +\left(1-t\right)\beta} \right)\le tf\left(
{\alpha} \right)+ \left(1-t\right) f\left( {\beta}
\right),\label{eq1.2}
\end{align}
for all points $\alpha,\beta \in J$ and all $t\in [0,1]$. If $-f$
is convex then we say that $f$ is concave. Moreover, if $f$ is
both convex and concave, then $f$ is said to be affine.

Geometrically, for two point $\left(x,f\left(x\right)\right)$ and
$\left(y,f\left(y\right)\right)$  on the graph of $f$ are on or
below the chord joining the endpoints  for all $x,y?\in I$, $x <
y$. In symbols, we write
\begin{align*}
f\left(t\right)\le   \frac{f\left( y \right)  - f\left( x \right)
}{y-x}   \left( {t-x} \right)+ f\left( x \right)
\end{align*}
for any $x \le t \le y$ and $x,y\in J$.

Equivalently, given a function $f : J\to \mathbb{R}$, we say that
$f$ admits a support line at $x \in J $ if there exists a $\lambda
\in \mathbb{R}$ such that
\begin{align}
f\left( t \right) \ge f\left( x \right) + \lambda \left( {t - x}
\right) \label{eq1.3}
\end{align}
for all $t\in J$.

The set of all such $\lambda$ is called the subdifferential of $f$
at $x$, and it's denoted by $\partial f$. Indeed, the
subdifferential gives us the slopes of the supporting lines for
the graph of $f$. So that if $f$ is convex then $\partial f(x) \ne
\emptyset$ at all interior points of its domain.

From this point of view  Abramovich \etal \cite{SJS} extend the
above idea for what they called superquadratic functions. Namely,
a function $f:[0,\infty)\to \mathbb{R}$ is called superquadratic
provided that for all $x\ge0$ there exists a constant $C_x\in
\mathbb{R}$ such that
\begin{align}
f\left( t \right) \ge f\left( x \right) + C_x \left( {t - x}
\right) + f\left( {\left| {t - x} \right|} \right)\label{eq1.4}
\end{align}
for all $t\ge0$. We say that $f$ is subquadratic if $-f$ is
superquadratic. Thus, for a superquadratic function we require
that $f$ lie above its tangent line plus a translation of $f$
itself.

Prima facie, superquadratic function  looks  to be stronger than
convex function itself but if $f$ takes negative values then it
may be considered as a weaker function. Therefore, if $f$ is
superquadratic and non-negative. Then $f$ is convex and increasing
\cite{SJS} (see also \cite{S}).

Moreover,   the following result holds for superquadratic
function.

\begin{lemma}\cite{SJS}
\label{lemma1}Let $f$ be superquadratic function. Then
\begin{enumerate}
    \item $f\left(0\right)\le 0$

    \item If $f$ is differentiable and $f(0)=f^{\prime}(0)=0$, then $C_x=f^{\prime}(x)$  for all $x\ge0$.

    \item If $f(x)\ge0$ for all $x\ge 0$, then $f$ is convex and $f(0)=f^{\prime}(0)=0$.
\end{enumerate}
\end{lemma}

The next result gives a sufficient condition when convexity
(concavity) implies super(sub)quaradicity.

\begin{lemma}\cite{SJS}
\label{lemma2}If $f^{\prime}$ is convex (concave) and
$f(0)=f^{\prime}(0)=0$, then is super(sub)quadratic. The converse
of is not true.
\end{lemma}

\begin{remark}
    Subquadraticity does always not imply concavity; i.e.,  there exists a subquadratic function which is convex. For example, $f(x)=x^p$, $x\ge 0$ and $1\le  p \le2$ is subquadratic and convex.
\end{remark}

\subsection{Popoviciu's inequality}
In 1906, Jensen in \cite{J} proved his famous characterization of
convex functions. Simply, for a continuous functions $f$ defined
on a real interval $I$, $f$ is convex if and only if
\begin{align*}
f\left( {\frac{{x + y}}{2}} \right) \le \frac{{f\left( x \right) +
        f\left( y \right)}}{2},
\end{align*}
for all $x,y\in I$.

In 1965, a parallel characterization of Jensen convexity was
presented by Popoviciu \cite{P} (for more details see \cite{NP},
p.6), where he proved his celebrated inequality, as follows:
\begin{theorem}
\label{thm2}    Let $f:I\to \mathbb{R}$ be continuous. Then, $f$
is convex if and
    only if
    \begin{align}
    \frac{2}{3}\left[{ f\left( {\frac{{x + z}}{2}} \right)+f\left(
        {\frac{{y + z}}{2}} \right)+f\left( {\frac{{x + y}}{2}}
        \right)}\right]
    \le f\left( {\frac{{x + y + z}}{3}} \right)
    +\frac{f\left( x \right)+f\left( y \right)+f\left( z \right)}{3}
    \label{eq1.5}
    \end{align}
    for all $x,y,z\in I$, and the equality occurred by $f(x)=x$, $x\in
    I$.
\end{theorem}

In fact, Popoviciu characterization of convex function is sound and several mathematicians greatly received his work since that time  and much of them considered  his characterization  as alternative approach to describe  convex functions.   For instance, the Popoviciu's inequality can be
considered as an elegant generalization of Hlawka's inequality
using convexity as a simple tool of geometry. 
Indeed, if $f(x)=|x|$, $x\in \mathbb{R}$, then the Popoviciu
inequality reduces to the famous Hlawka  inequality, which reads:
\begin{align*}
\label{Hlawkaineq}\left| x \right| + \left| y \right| + \left| z
\right| + \left| {x + y + z} \right| \ge \left| {x + z} \right| +
\left| {z + y} \right| + \left| {x + y} \right|.
\end{align*}
Geometrically,  Hlawka  inequality means  the total length over
all sums of pairs from three vectors is not greater than the
perimeter of the quadrilateral defined by the three vectors.  This
geometric meaning was given by D. Smiley \& M. Smiley \cite{W}
(see also \cite{SS}, p. 756).  For other related results see
\cite{MPF} and \cite{NF}.

Also, The extended version of Hlawka's
inequality  to
several variables was not possible without the help of Popoviciu's inequality, as it inspired  the authors of \cite{BNP} to develop a
higher dimensional analogue of Popoviciu's inequality based on his
characterization. Interesting generalizations and counterparts of
Popoviciu inequality with some ramified consequences can be
found in \cite{G}, \cite{TTH}, \cite{TTMT} and \cite{VS}.\\

Recenty, The corresponding version of  Popoviciu inequality for
${\rm{GG}}$-convex (Recall that: a positive real valued function $f$ is ${\rm{GG}}$-convex if and only if $
f\left( {x^t y^{1 - y} } \right) \le \left[ {f\left( x \right)} \right]^t \left[ {f\left( y \right)} \right]^{1 - t} $ for all $t\in [0,1]$ and all $x,y \ge0$)  was discussed eighteen
years ago by Niculescu in \cite{N}, where he  proved that for
all $x,y,z\in I \subset [0,\infty)$, the inequality
\begin{align*}
f^2\left( {\sqrt {xz}} \right)f^2\left( {\sqrt {yz}}
\right)f^2\left( {\sqrt {xy}} \right) \le   f^3\left(
{\sqrt[3]{xyz} } \right) f\left( x \right)f\left( y \right)f\left(
z \right)
\end{align*}
holds for all $x,y,z\in I$.

 Seeking the operator version of Popoviciu's inequality \eqref{eq1.5}, the expected version of \eqref{eq1.5} for selfadjoint operators is
\begin{multline*}
\frac{2}{3}\left[f\left( {\left\langle { \frac{{A + B}}{2} u,u} \right\rangle } \right) + f\left(
{\left\langle { \frac{{B+D}}{2} u,u}\right\rangle } \right)+ f\left( {\left\langle { \frac{{A + D}}{2} u,u} \right\rangle }
\right)\right] 
\\
\le \left\langle { \frac{{f\left( A \right) + f\left(B \right) + f\left( D \right)}}{3} u,u} \right\rangle+f\left( {\left\langle { \frac{{A + B + D}}{3}u,u} \right\rangle } \right)
\end{multline*}
for every selfadjoint operators $A,B,D\in \mathcal{B}\left(\mathcal{H} \right)$ whose spectra contained in $I$ and every convex function  $f$ defined on $I$ and this is valid for each $u\in \mathcal{K}$ with $\|u\| =1$. The proof of the above inequality is obvious by taking $x=\left\langle { A u,u} \right\rangle$, $y=\left\langle { B u,u} \right\rangle$ and $z=\left\langle { D u,u} \right\rangle$ in \eqref{eq1.5}.

In this work, we offer two operator versions of 
Popoviciu's inequality for positive selfadjoint
operators in Hilbert spaces under positive linear maps for
both superquadratic and convex functions with some other related results.

\section{Main Result}

Throughout this work and in all needed situations, $f$ is real
valued continuous function defined on $\left[0,\infty\right)$. In
order to prove our main result, we need the following result 
concerning    Jensen's inequality for superquadratic functions.
Let us don't miss the chance here to mention that the next
result was proved in   \cite{KS} and originally in \cite{K}  for
positive selfadjoint $(n \times n)$--matrices with complex entries
under unital completely positive linear maps. However, let us
state down this result in more general Hilbert spaces for
normalized   positive linear maps.
\begin{theorem}
\label{thm1}Let $A\in \mathcal{B}\left(\mathcal{H} \right)$ be a
positive selfadjoint operator, $\Phi:\mathcal{B}\left(\mathcal{H}
\right)\to \mathcal{B}\left(\mathcal{K} \right)$  be a normalized
positive linear map. If   $f:\left[0,\infty\right)\to \mathbb{R}$
is super(sub)quadratic, then we have
\begin{align} 
\left\langle {\Phi \left( {f\left( A \right)}
\right)x,x} \right\rangle \ge (\le)  f\left( {\left\langle {\Phi
\left( {A} \right)x,x} \right\rangle } \right)   + \left\langle
{\Phi \left( {f\left( {\left| {A - \left\langle
{\Phi\left(A\right)x,x} \right\rangle 1_{\mathcal{H}} } \right|}
\right)} \right)x,x} \right\rangle
\end{align}
for every $x\in \mathcal{K}$ with $\|x\|=1$.
\end{theorem}

For more recent results concerning inequalities for selfadjoint operatos and other related result, we suggest
\cite{SIP}, \cite{AD}--\cite{FMPS}, \cite{KLPP}, \cite{MPP}   and
\cite{MP}.

The  operator version of Popoviciu's inequality for superquadratic
functions under positive linear maps is proved in the next result.
\begin{theorem}
    \label{thm2.1}Let $A,B,C\in \mathcal{B}\left(\mathcal{H} \right)$ be three positive selfadjoint operators, $\Phi:\mathcal{B}\left(\mathcal{H} \right)\to \mathcal{B}\left(\mathcal{K} \right)$  be a normalized positive linear map. If  $f:\left[0,\infty\right)\to \mathbb{R}$    is superquadratic, then we have
    \begin{align}
    &\left\langle {\Phi \left( {\frac{{f\left( A \right) + f\left(
                    B \right) + f\left( D \right)}}{3}} \right)x,x} \right\rangle
    +f\left( {\left\langle {\Phi \left( {\frac{{A + B + D}}{3}}
            \right)x,x} \right\rangle } \right)
    \label{eq2.3}\\
    &\ge \frac{2}{3}\left[f\left( {\left\langle {\Phi \left(
            {\frac{{A + B}}{2}} \right)x,x} \right\rangle } \right) + f\left(
    {\left\langle {\Phi \left( {\frac{{B+D}}{2}} \right)x,x}
        \right\rangle } \right)+ f\left( {\left\langle {\Phi \left(
            {\frac{{A + D}}{2}} \right)x,x} \right\rangle }
    \right)\right]\nonumber
    \\
    &\qquad+ \frac{1}{3} \left[\left\langle {\Phi \left( {f\left(
            {\left| {A -
                    \left\langle {\Phi\left(\frac{{B + D}}{2}\right)x,x}
                    \right\rangle 1_{\mathcal{H}} } \right|} \right)} \right)x,x}
    \right\rangle +     f\left( {\left| {\left\langle
            {\Phi \left( {\frac{{2A - B - D}}{6}} \right)x,x} \right\rangle
        } \right|} \right)\right.
    \nonumber\\
    &\qquad\qquad+\left\langle {\Phi \left( {f\left( {\left| {D -
                    \left\langle {\Phi\left(\frac{{A + B}}{2}\right)x,x}
                    \right\rangle 1_{\mathcal{H}} } \right|} \right)} \right)x,x}
    \right\rangle+  f\left( {\left| {\left\langle {\Phi
                \left( {\frac{{2D - A - B}}{6}} \right)x,x} \right\rangle }
        \right|} \right)
    \nonumber\\
    &\qquad\qquad\left.  + \left\langle {\Phi \left( {f\left( {\left|
                {B -
                    \left\langle {\Phi\left(\frac{{A + D}}{2}\right)x,x}
                    \right\rangle 1_{\mathcal{H}} } \right|} \right)} \right)x,x}
    \right\rangle + f\left( {\left| {\left\langle
            {\Phi \left( {\frac{{2B - A - D}}{6}} \right)x,x} \right\rangle
        } \right|} \right)\right]\nonumber
    \end{align}
    for each $x\in \mathcal{K}$ with $\|x\| =1$.\\
\end{theorem}

\begin{proof}
	Since $f$ is  superquadratic  on $I$, then
	by utilizing the continuous functional calculus for the operator
	$E\ge 0$ we have by the property \eqref{eq1.1} for the inequality
	\eqref{eq1.4}  we have
	\begin{align*}
	f\left( E \right) \ge f\left( s \right)\cdot 1_{\mathcal{H}} +
	C_s \left( {E -s \cdot 1_{\mathcal{H}}} \right) +
	f\left( {\left| {E - s\cdot 1_{\mathcal{H}}} \right|} \right).
	\end{align*}
	and since $\Phi$ is normalized positive linear map we get
	\begin{align*}
	\Phi\left(f\left( E \right)\right) \ge  f\left( s \right)\cdot
	1_{\mathcal{K}}   + C_s  \Phi\left(E -s \cdot
	1_{\mathcal{H}} \right) + \Phi\left(f\left( {\left| {E - s\cdot
			1_{\mathcal{H}}} \right|} \right)\right)
	\end{align*}
	and this  implies that
	\begin{align}
	\left\langle {\Phi \left( {f\left( E \right)} \right)x,x}
	\right\rangle  \ge f\left( {s} \right) \left\langle { x,x}
	\right\rangle  + C_s \left\langle {\left[ {\Phi \left( {E - s
				\cdot 1_{\mathcal{H}} } \right)} \right]x,x} \right\rangle  +
	\left\langle {\Phi \left( {f\left( {\left| {E - s \cdot
					1_{\mathcal{H}}  } \right|} \right)} \right)x,x} \right\rangle
	\label{eq2.2}
	\end{align}
	for each vector $x\in \mathcal{K}$ with $\|x\|=1$.
	
    Let $A,B,D$ be three positive selfadjoint operators in $
    \mathcal{B}\left(\mathcal{H} \right)$. Since $f$ is superquadratic
    then by applying \eqref{eq2.2}  for the operator $A\ge0$ with $s_1=
    \left\langle {\Phi\left(\frac{{B + D}}{2}\right)x,x}
    \right\rangle$, we get
    \begin{align}
    \left\langle {\Phi \left( {f\left( A \right)} \right)x,x}
    \right\rangle &\ge f\left( {\left\langle {\Phi \left( {\frac{{B +
                        D}}{2}} \right)x,x} \right\rangle } \right) +C_{s_1} \left\langle {\Phi \left( {\frac{{2A
                    - B - D}}{2}} \right)x,x} \right\rangle  \label{eq2.4}
    \\
    &\qquad+ \left\langle {\Phi \left( {f\left( {\left| {A -
                    \left\langle {\Phi \left( {\frac{{B + D}}{2}} \right)x,x}
                    \right\rangle 1_{\mathcal{H}} } \right|} \right)} \right)x,x}
    \right\rangle \nonumber
    \end{align}
    for each $x\in \mathcal{K}$ with $\|x\| =1$.

    Again applying \eqref{eq2.2}  for the operator  $D\ge0$ with $s_2=
    \left\langle {\Phi\left(\frac{{A + B}}{2}\right)x,x}
    \right\rangle$
    \begin{align}
    \left\langle {\Phi \left( {f\left( D \right)} \right)x,x}
    \right\rangle &\ge f\left( {\left\langle {\Phi \left( {\frac{{A +
                        B}}{2}} \right)x,x} \right\rangle } \right) + C_{s_2} \left\langle {\Phi \left(
        {\frac{{2D-A-B}}{2}} \right)x,x} \right\rangle \label{eq2.5}
    \\
    &\qquad + \left\langle {\Phi \left( {f\left( {\left| {D -
                    \left\langle {\Phi \left( {\frac{{A + B}}{2}} \right)x,x}
                    \right\rangle 1_{\mathcal{H}} } \right|} \right)} \right)x,x}
    \right\rangle  \nonumber
    \end{align}
    for each $x\in \mathcal{K}$ with $\|x\| =1$.

    Also, for the operator  $B\ge0$ with $s_3= \left\langle
    {\Phi\left(\frac{{A + D}}{2}\right)x,x} \right\rangle$
    \begin{align}
    \left\langle {\Phi \left( {f\left(B \right)} \right)x,x}
    \right\rangle &\ge f\left( {\left\langle {\Phi \left( {\frac{{A +
                        D}}{2}} \right)x,x} \right\rangle } \right) + C_{s_3} \left\langle {\Phi \left( { \frac{{2B
                    - A - D}}{2} } \right)x,x} \right\rangle \label{eq2.6}
    \\
    &\qquad + \left\langle {\Phi \left( {f\left( {\left| {B -
                    \left\langle {\Phi \left( {\frac{{A + D}}{2}} \right)x,x}
                    \right\rangle 1_{\mathcal{H}} } \right|} \right)} \right)x,x}
    \right\rangle. \nonumber
    \end{align}
    for each $x\in \mathcal{K}$ with $\|x\| =1$.

    Adding the inequalities \eqref{eq2.4}--\eqref{eq2.6} and then
    multiplying by $\frac{1}{3}$ we get
    \begin{align}
    &\left\langle {\Phi \left( {\frac{{f\left( A \right) + f\left( B
                    \right) + f\left( D \right)}}{3}} \right)x,x} \right\rangle
    \nonumber\\
    &\ge \frac{1}{3}\left[f\left( {\left\langle {\Phi \left( {\frac{{A
                        + B}}{2}} \right)x,x} \right\rangle } \right) + f\left(
    {\left\langle {\Phi \left( {\frac{{B+D}}{2}} \right)x,x}
        \right\rangle } \right)+ f\left( {\left\langle {\Phi \left(
            {\frac{{A + D}}{2}} \right)x,x} \right\rangle } \right)\right]
    \label{eq2.7}\\
    &\qquad+  \frac{1}{3}\left[C_{s_1} \left\langle {\Phi \left(
        {\frac{{2A - B-D}}{2}} \right)x,x} \right\rangle +C_{s_2} \left\langle
    {\Phi \left( {\frac{{2D-A-B}}{2}} \right)x,x} \right\rangle
    \right.
    \nonumber\\
    &\qquad\qquad+ C_{ s_3} \left\langle {\Phi \left( {
            \frac{{2B - A - D}}{2} } \right)x,x} \right\rangle
    \nonumber\\
    &\qquad+ \left\langle {\Phi \left( {f\left( {\left| {A -
                    \left\langle {\Phi\left(\frac{{B + D}}{2}\right)x,x} \right\rangle
                    1_{\mathcal{H}} } \right|} \right)} \right)x,x} \right\rangle +
    \left\langle {\Phi \left( {f\left( {\left| {D - \left\langle
                    {\Phi\left(\frac{{A+B}}{2}\right)x,x} \right\rangle
                    1_{\mathcal{H}} } \right|} \right)} \right)x,x} \right\rangle
    \nonumber\\
    &\qquad\qquad\left. + \left\langle {\Phi \left( {f\left( {\left|
                {B - \left\langle {\Phi\left(\frac{{A + D}}{2}\right)x,x}
                    \right\rangle 1_{\mathcal{H}} } \right|} \right)} \right)x,x}
    \right\rangle\right].\nonumber
    \end{align}
    Setting $C:=\min\left\{C_{s_1},C_{s_2},C_{s_3}\right\}$, then \eqref{eq2.7} reduces to
    \begin{align}
    &\left\langle {\Phi \left( {\frac{{f\left( A \right) + f\left( B
                    \right) + f\left( D \right)}}{3}} \right)x,x} \right\rangle
    \nonumber\\
    &\ge \frac{1}{3}\left[f\left( {\left\langle {\Phi \left( {\frac{{A
                        + B}}{2}} \right)x,x} \right\rangle } \right) + f\left(
    {\left\langle {\Phi \left( {\frac{{B+D}}{2}} \right)x,x}
        \right\rangle } \right)+ f\left( {\left\langle {\Phi \left(
            {\frac{{A + D}}{2}} \right)x,x} \right\rangle } \right)\right]
    \nonumber\\
    &\qquad+  \frac{1}{3}C\left[ \left\langle {\Phi \left( {\frac{{2A
                    - B-D}}{2}} \right)x,x} \right\rangle +  \left\langle {\Phi \left(
        {\frac{{2D-A-B}}{2}} \right)x,x} \right\rangle
    + \left\langle {\Phi \left( { \frac{{2B - A - D}}{2} } \right)x,x} \right\rangle \right.
    \nonumber\\
    &\qquad+ \left\langle {\Phi \left( {f\left( {\left| {A -
                    \left\langle {\frac{{B + D}}{2}x,x} \right\rangle 1_{\mathcal{H}}
                } \right|} \right)} \right)x,x} \right\rangle + \left\langle {\Phi
        \left( {f\left( {\left| {D - \left\langle {\frac{{A+B}}{2}x,x}
                    \right\rangle 1_{\mathcal{H}} } \right|} \right)} \right)x,x}
    \right\rangle
    \nonumber\\
    &\qquad\left. + \left\langle {\Phi \left( {f\left( {\left| {B -
                    \left\langle {\frac{{A+D}}{2}x,x} \right\rangle 1_{\mathcal{H}} }
                \right|} \right)} \right)x,x} \right\rangle\right]\nonumber
    \\
    &=\frac{1}{3}\left[f\left( {\left\langle {\Phi \left( {\frac{{A +
                        B}}{2}} \right)x,x} \right\rangle } \right) + f\left(
    {\left\langle {\Phi \left( {\frac{{B+D}}{2}} \right)x,x}
        \right\rangle } \right)+ f\left( {\left\langle {\Phi \left(
            {\frac{{A + D}}{2}} \right)x,x} \right\rangle } \right)\right]
    \label{eq2.8}\\
    &\qquad+\frac{1}{3} \left[ \left\langle {\Phi \left( {f\left(
            {\left| {A -
                    \left\langle {\frac{{B + D}}{2}x,x} \right\rangle 1_{\mathcal{H}}
                } \right|} \right)} \right)x,x} \right\rangle + \left\langle
    {\Phi \left( {f\left( {\left| {D - \left\langle
                    {\frac{{A+B}}{2}x,x} \right\rangle 1_{\mathcal{H}} } \right|}
            \right)} \right)x,x} \right\rangle\right.
    \nonumber\\
    &\qquad\left. + \left\langle {\Phi \left( {f\left( {\left| {B -
                    \left\langle {\frac{{A+D}}{2}x,x} \right\rangle 1_{\mathcal{H}} }
                \right|} \right)} \right)x,x} \right\rangle\right].\nonumber
    \end{align}
    Now, applying \eqref{eq1.4} three times for $t= {\left\langle
        {\Phi \left( {\frac{{A + B + C}}{3}} \right)x,x} \right\rangle }$
    with $s_1,s_2,s_3$,  then for each unit vector  $x\in \mathcal{K}$, we
    get respectively,
    \begin{align}
    & f\left( {\left\langle {\Phi \left( {\frac{{A + B + D}}{3}}
            \right)x,x} \right\rangle } \right)
    \nonumber\\
    &\ge f\left( {\left\langle {\Phi \left( {\frac{{A + B}}{2}}
            \right)x,x} \right\rangle } \right)
    +  C_{s_1} \left\langle {\Phi \left( {\frac{{A + B + D}}{3}} \right)x,x} \right\rangle  - C_{s_1} \left\langle {\Phi \left( {\frac{{A + B}}{2}} \right)x,x} \right\rangle
    \label{eq2.9}\\
    &\qquad+ f\left( {\left| {\left\langle {\Phi \left( {\frac{{A + B
                            + D}}{3}} \right)x,x} \right\rangle  - \left\langle {\Phi \left(
                {\frac{{A + B}}{2}} \right)x,x} \right\rangle } \right|} \right)
    \nonumber
    \\
    &=
    f\left( {\left\langle {\Phi \left( {\frac{{A + B}}{2}} \right)x,x} \right\rangle } \right) +  C_{s_1} \left\langle {\Phi \left( {\frac{{2D - A - B}}{6}} \right)x,x} \right\rangle  +  f\left( {\left| {\left\langle {\Phi \left( {\frac{{2D - A - B}}{6}} \right)x,x} \right\rangle } \right|} \right),\nonumber
    \end{align}
    \begin{align}
    & f\left( {\left\langle {\Phi \left( {\frac{{A + B +D}}{3}}
            \right)x,x} \right\rangle } \right)
    \nonumber\\
    &\ge f\left( {\left\langle {\Phi \left( {\frac{{A + B}}{2}}
            \right)x,x} \right\rangle } \right) +  C_{s_2}\left\langle {\Phi \left( {\frac{{A + B + D}}{3}}
        \right)x,x} \right\rangle  - C_{s_2}\left\langle {\Phi \left( {\frac{{B+D}}{2}} \right)x,x}
    \right\rangle
    \label{eq2.10}\\
    &\qquad+ f\left( {\left| {\left\langle {\Phi \left( {\frac{{A + B+ D}}{3}} \right)x,x} \right\rangle  - \left\langle {\Phi \left(
                {\frac{{B+D}}{2}} \right)x,x} \right\rangle } \right|} \right)
    \nonumber
    \\
    &= f\left( {\left\langle {\Phi \left( {\frac{{A + B}}{2}}
            \right)x,x} \right\rangle } \right) +  C_{s_2} \left\langle {\Phi \left( {\frac{{2A - B -
                    D}}{6}} \right)x,x} \right\rangle  +  f\left( {\left|
        {\left\langle {\Phi \left( {\frac{{2A - B - D}}{6}} \right)x,x}
            \right\rangle } \right|} \right)\nonumber
    \end{align}
    and
    \begin{align}
    & f\left( {\left\langle {\Phi \left( {\frac{{A + B + D}}{3}}
            \right)x,x} \right\rangle } \right)
    \nonumber\\
    &\ge f\left( {\left\langle {\Phi \left( {\frac{{A + D}}{2}}
            \right)x,x} \right\rangle } \right) +  C_{s_3} \left\langle {\Phi \left( {\frac{{A + B + D}}{3}}
        \right)x,x} \right\rangle  - C_{s_3} \left\langle {\Phi \left( {\frac{{A+D}}{2}} \right)x,x}
    \right\rangle
    \label{eq2.11}\\
    &\qquad+ f\left( {\left| {\left\langle {\Phi \left( {\frac{{A + B +D}}{3}} \right)x,x} \right\rangle  - \left\langle {\Phi \left(
                {\frac{{A+D}}{2}} \right)x,x} \right\rangle } \right|} \right)
    \nonumber
    \\
    &= f\left( {\left\langle {\Phi \left( {\frac{{A + D}}{2}}
            \right)x,x} \right\rangle } \right) +  C_{s_3}\left\langle {\Phi \left( {\frac{{2B - A - D}}{6}}
        \right)x,x} \right\rangle  +  f\left( {\left| {\left\langle {\Phi
                \left( {\frac{{2B - A - D}}{6}} \right)x,x} \right\rangle }
        \right|} \right).\nonumber
    \end{align}
    Multiplying each inequality by $\frac{1}{3}$ and summing up the
    inequalities \eqref{eq2.9}--\eqref{eq2.11}, we get
    \begin{align}
    &f\left( {\left\langle {\Phi \left( {\frac{{A + B + D}}{3}} \right)x,x} \right\rangle } \right)
    \nonumber\\
    &\ge \frac{1}{3}f\left( {\left\langle {\Phi \left( {\frac{{A + B}}{2}} \right)x,x} \right\rangle } \right) + \frac{1}{3}f\left( {\left\langle {\Phi \left( {\frac{{A + D}}{2}} \right)x,x} \right\rangle } \right) + \frac{1}{3}f\left( {\left\langle {\Phi \left( {\frac{{A + B}}{2}} \right)x,x} \right\rangle } \right)\label{eq2.12}
    \\
    &\qquad+ \frac{1}{3} C_{s_1} \left\langle {\Phi \left( {\frac{{2D - A - B}}{6}} \right)x,x} \right\rangle  + \frac{1}{3} f\left( {\left| {\left\langle {\Phi \left( {\frac{{2C - A - B}}{6}} \right)x,x} \right\rangle } \right|} \right)
    \nonumber\\
    &\qquad+\frac{1}{3}C_{s_2} \left\langle {\Phi \left( {\frac{{2A - B - D}}{6}} \right)x,x} \right\rangle  + \frac{1}{3} f\left( {\left| {\left\langle {\Phi \left( {\frac{{2A - B - D}}{6}} \right)x,x} \right\rangle } \right|} \right)\nonumber
    \\
    &\qquad+ \frac{1}{3} C_{s_3} \left\langle {\Phi \left( {\frac{{2B - A - D}}{6}} \right)x,x} \right\rangle  +\frac{1}{3}  f\left( {\left| {\left\langle {\Phi \left( {\frac{{2B - A - D}}{6}} \right)x,x} \right\rangle } \right|} \right)\nonumber
    \end{align}
    But since $C:=\min\left\{C_{s_1},C_{s_2},C_{s_3}\right\}$, then \eqref{eq2.12}
    becomes
    \begin{align}
    &f\left( {\left\langle {\Phi \left( {\frac{{A + B + D}}{3}}
            \right)x,x} \right\rangle } \right)
    \nonumber\\
    &\ge \frac{1}{3}f\left( {\left\langle {\Phi \left( {\frac{{A +
                        B}}{2}} \right)x,x} \right\rangle } \right) + \frac{1}{3}f\left(
    {\left\langle {\Phi \left( {\frac{{A + D}}{2}} \right)x,x}
        \right\rangle } \right) + \frac{1}{3}f\left( {\left\langle {\Phi
            \left( {\frac{{A + B}}{2}} \right)x,x} \right\rangle } \right)
    \nonumber\\
    &\qquad+ \frac{1}{3} C \left[\left\langle {\Phi \left( {\frac{{2D
                    - A - B}}{6}} \right)x,x} \right\rangle  + \left\langle {\Phi
        \left( {\frac{{2A - B - D}}{6}} \right)x,x} \right\rangle +
    \left\langle {\Phi \left( {\frac{{2B - A - D}}{6}} \right)x,x}
    \right\rangle  \right]
    \nonumber\\
    &\qquad+ \frac{1}{3}  f\left( {\left| {\left\langle {\Phi \left(
                {\frac{{2D - A - B}}{6}} \right)x,x} \right\rangle } \right|}
    \right) + \frac{1}{3} f\left( {\left| {\left\langle {\Phi \left(
                {\frac{{2A - B - D}}{6}} \right)x,x} \right\rangle } \right|}
    \right)
    \nonumber\\
    &\qquad+ \frac{1}{3}  f\left( {\left| {\left\langle {\Phi \left(
                {\frac{{2B - A - D}}{6}} \right)x,x} \right\rangle } \right|}
    \right)
    \nonumber\\
    &=\frac{1}{3}f\left( {\left\langle {\Phi \left( {\frac{{A +
                        B}}{2}} \right)x,x} \right\rangle } \right) + \frac{1}{3}f\left(
    {\left\langle {\Phi \left( {\frac{{A + D}}{2}} \right)x,x}
        \right\rangle } \right)\label{eq2.13}
    \\
    &\qquad + \frac{1}{3}f\left( {\left\langle {\Phi \left( {\frac{{A
                        + B}}{2}} \right)x,x} \right\rangle } \right) + \frac{1}{3}
    f\left( {\left| {\left\langle {\Phi \left( {\frac{{2D - A -
                            B}}{6}} \right)x,x} \right\rangle } \right|} \right)
    \nonumber\\
    &\qquad+\frac{1}{3}  f\left( {\left| {\left\langle {\Phi \left(
                {\frac{{2A - B - D}}{6}} \right)x,x} \right\rangle } \right|}
    \right) +\frac{1}{3}f\left( {\left| {\left\langle {\Phi \left(
                {\frac{{2B - A - D}}{6}} \right)x,x} \right\rangle } \right|}
    \right).\nonumber
    \end{align}
    Adding the inequalities \eqref{eq2.8} and \eqref{eq2.13} we get
    that
    \begin{align*}
    &\left\langle {\Phi \left( {\frac{{f\left( A \right) + f\left( B
                    \right) + f\left( D \right)}}{3}} \right)x,x} \right\rangle
    +f\left( {\left\langle {\Phi \left( {\frac{{A + B + D}}{3}}
            \right)x,x} \right\rangle } \right)
    \nonumber\\
    &\ge \frac{2}{3}\left[f\left( {\left\langle {\Phi \left( {\frac{{A
                        + B}}{2}} \right)x,x} \right\rangle } \right) + f\left(
    {\left\langle {\Phi \left( {\frac{{B+D}}{2}} \right)x,x}
        \right\rangle } \right)+ f\left( {\left\langle {\Phi \left(
            {\frac{{A + D}}{2}} \right)x,x} \right\rangle }
    \right)\right]\nonumber
    \\
    &\qquad+ \frac{1}{3} \left[\left\langle {\Phi \left( {f\left(
            {\left| {A - \left\langle {\Phi\left(\frac{{B + D}}{2}\right)x,x}
                    \right\rangle 1_{\mathcal{H}} } \right|} \right)} \right)x,x}
    \right\rangle + \left\langle {\Phi \left( {f\left( {\left| {D -
                    \left\langle {\Phi\left(\frac{{A + B}}{2}\right)x,x} \right\rangle
                    1_{\mathcal{H}} } \right|} \right)} \right)x,x}
    \right\rangle\right.
    \nonumber\\
    &\qquad\qquad  + \left\langle {\Phi \left( {f\left( {\left| {B -
                    \left\langle {\Phi\left(\frac{{A + D}}{2}\right)x,x} \right\rangle
                    1_{\mathcal{H}} } \right|} \right)} \right)x,x} \right\rangle+
    f\left( {\left| {\left\langle {\Phi
                \left( {\frac{{2D - A - B}}{6}} \right)x,x} \right\rangle }
        \right|} \right) \nonumber
    \\
    &\qquad\qquad\left.+   f\left( {\left| {\left\langle {\Phi \left(
                {\frac{{2A - B - D}}{6}} \right)x,x} \right\rangle } \right|}
    \right) + f\left( {\left| {\left\langle {\Phi \left( {\frac{{2B -
                            A - D}}{6}} \right)x,x} \right\rangle } \right|}
    \right)\right]\nonumber
    \end{align*}
    for each $x\in \mathcal{K}$ with $\|x\| =1$, which gives the
    required inequality in \eqref{eq2.3}.
\end{proof}

\begin{corollary}
    \label{cor1}Let $A,B,C\in \mathcal{B}\left(\mathcal{H} \right)$ be three positive selfadjoint operators, $\Phi:\mathcal{B}\left(\mathcal{H} \right)\to \mathcal{B}\left(\mathcal{K} \right)$  be a normalized positive linear map. If $f:\left[0,\infty\right)\to \mathbb{R}$ is subquadratic, then we have
    \begin{align}
    &\left\langle {\Phi \left( {\frac{{f\left( A \right) + f\left(
                    B \right) + f\left( D \right)}}{3}} \right)x,x} \right\rangle
    +f\left( {\left\langle {\Phi \left( {\frac{{A + B + D}}{3}}
            \right)x,x} \right\rangle } \right)
    \label{eq2.14}\\
    &\le \frac{2}{3}\left[f\left( {\left\langle {\Phi \left(
            {\frac{{A + B}}{2}} \right)x,x} \right\rangle } \right) + f\left(
    {\left\langle {\Phi \left( {\frac{{B+D}}{2}} \right)x,x}
        \right\rangle } \right)+ f\left( {\left\langle {\Phi \left(
            {\frac{{A + D}}{2}} \right)x,x} \right\rangle }
    \right)\right] \nonumber
    \\
    &\qquad+ \frac{1}{3} \left[\left\langle {\Phi \left( {f\left(
            {\left| {A -
                    \left\langle {\Phi\left(\frac{{B + D}}{2}\right)x,x}
                    \right\rangle 1_{\mathcal{H}} } \right|} \right)} \right)x,x}
    \right\rangle +     f\left( {\left| {\left\langle
            {\Phi \left( {\frac{{2A - B - D}}{6}} \right)x,x} \right\rangle
        } \right|} \right)\right.
    \nonumber\\
    &\qquad\qquad+\left\langle {\Phi \left( {f\left( {\left| {D -
                    \left\langle {\Phi\left(\frac{{A + B}}{2}\right)x,x}
                    \right\rangle 1_{\mathcal{H}} } \right|} \right)} \right)x,x}
    \right\rangle+  f\left( {\left| {\left\langle {\Phi
                \left( {\frac{{2D - A - B}}{6}} \right)x,x} \right\rangle }
        \right|} \right)
    \nonumber\\
    &\qquad\qquad\left.  + \left\langle {\Phi \left( {f\left( {\left|
                {B -
                    \left\langle {\Phi\left(\frac{{A + D}}{2}\right)x,x}
                    \right\rangle 1_{\mathcal{H}} } \right|} \right)} \right)x,x}
    \right\rangle + f\left( {\left| {\left\langle
            {\Phi \left( {\frac{{2B - A - D}}{6}} \right)x,x} \right\rangle
        } \right|} \right)\right]\nonumber
    \end{align}
    for each $x\in \mathcal{K}$ with $\|x\| =1$.\\
\end{corollary}

\begin{proof}
    Repeating the same steps in the proof of Theorem \ref{thm2.1}, by
    writing $`\le$' instead of $`\ge$' and in this case we consider
    $C:=\max\left\{C_{s_1},C_{s_2},C_{s_3}\right\}$,
\end{proof}
A generalization of the result in Theorem \ref{thm1} is deduced
as follows:
\begin{corollary}
\label{cor2}Let $A\in \mathcal{B}\left(\mathcal{H} \right)$ be a
positive selfadjoint operator, $\Phi:\mathcal{B}\left(\mathcal{H}
\right)\to \mathcal{B}\left(\mathcal{K} \right)$  be a normalized
positive linear map.  If  $f:\left[0,\infty\right)\to \mathbb{R}$
is super(sub)quadratic, then we have
\begin{align}
f\left( {\left\langle {\Phi \left( {A} \right)x,x} \right\rangle }
\right)   \le (\ge) \left\langle {\Phi \left( {f\left( A \right)}
	\right)x,x} \right\rangle - \left\langle {\Phi \left( {f\left(
		{\left| {A - \left\langle {\Phi\left(A\right)x,x} \right\rangle 1_{\mathcal{H}} }
			\right|} \right)} \right)x,x} \right\rangle - f\left( {0} \right) \label{eq2.15}
\end{align}
for each $x\in \mathcal{K}$ with $\|x\| =1$. 
\end{corollary}

\begin{proof}
Setting  $D=B=A$ in \eqref{eq2.3} we get the required result.
\end{proof}
 
\begin{remark}
According to Corollary \ref{cor2}, if $f\ge0$ $(f\le 0)$ then the above inequality refines
and improves Theorem \ref{thm1}.
\end{remark} 	

The classical Bohr inequality for scalars reads that if $a$, $b$ are complex numbers and
$p,q>1$ with $\frac{1}{p}+\frac{1}{q}=1$, then
\begin{align*}
\left|a-b\right|^2\le p \left|a\right|^2+q\left|b\right|^2.
\end{align*}
The first result regarding operator version of Bohr inequality was established in \cite{H}. For refinements, generalizations and other related results see \cite{CHP}, \cite{CP}, \cite{FZ} and \cite{Z}. 

 The following  Bohr's type inequalities for positive selfadjoint
operators under positive linear maps are hold:
\begin{corollary}
Let $A\in \mathcal{B}\left(\mathcal{H} \right)$ be a positive
selfadjoint operator, $\Phi:\mathcal{B}\left(\mathcal{H}
\right)\to \mathcal{B}\left(\mathcal{K} \right)$  be a normalized
positive linear map.
\begin{enumerate}
\item If  $f:\left[0,\infty\right)\to \mathbb{R}$    is
superquadratic, then we have
\begin{align*}
\left\| {\Phi \left( {f\left( {\left| {A - \left\| {\Phi \left( A
\right)} \right\|1_H } \right|} \right)} \right)} \right\| \le
\left\| {\Phi \left( {f\left( A \right)} \right)} \right\| -
f\left( {\left\| {\Phi \left( A \right)} \right\|} \right) -
f\left( 0 \right).
\end{align*}
In particular, let $f(t)=t^r$, $r\ge2$, $t\ge0$.
\begin{align*}
\left\| {\Phi \left( { \left| {A - \left\| {\Phi \left( A \right)}
\right\|1_H } \right|^r } \right)} \right\| \le \left\| {\Phi
\left( {A^r} \right)} \right\| -  \left\| {\Phi \left( A \right)}
\right\|^r.
\end{align*}

\item If  $f:\left[0,\infty\right)\to \mathbb{R}$    is
subquadratic, then we have
\begin{align*}
\left\| {\Phi \left( {f\left( {\left| {A - \left\| {\Phi \left( A
\right)} \right\|1_H } \right|} \right)} \right)} \right\| \ge
\left\| {\Phi \left( {f\left( A \right)} \right)} \right\| -
f\left( {\left\| {\Phi \left( A \right)} \right\|} \right) -
f\left( 0 \right).
\end{align*}
In particular, let $f(t)=t^r$, $0<r\le2$, $t\ge0$.
\begin{align*}
\left\| {\Phi \left( { \left| {A - \left\| {\Phi \left( A \right)}
\right\|1_H } \right|^r } \right)} \right\| \ge \left\| {\Phi
\left( {A^r} \right)} \right\| -  \left\| {\Phi \left( A \right)}
\right\|^r.
\end{align*}
\end{enumerate}

\end{corollary}
\begin{proof}
Taking the supremum in \eqref{eq2.15} over $x\in \mathcal{K}$ with
$\|x\|=1$ we obtain the required result(s).
\end{proof}

\begin{corollary}
    \label{cor5}Let $A_j,B_j,D_j\in \mathcal{B}\left(\mathcal{H} \right)$ be three positive selfadjoint operators for every $j=1,\cdots,n$. Let $\Phi_j:\mathcal{B}\left(\mathcal{H} \right)\to \mathcal{B}\left(\mathcal{K} \right)$  be  positive linear maps such that $\sum\limits_{j = 1}^n {\Phi _j \left( {1_H } \right)}  = 1_K$.
    If $f:\left[0,\infty\right)\to \mathbb{R}$ is superquadratic, then we have
    \begin{align}
    &{f\left(\left\langle {\sum\limits_{j =
                1}^n\Phi_j\left(\frac{A_j+B_j+D_j}{3}\right) u,u}
        \right\rangle\right)}+ {\left\langle { \sum\limits_{j =
                1}^n\Phi_j\left(\frac{{f\left( A_j \right) +f\left( B_j \right) +
                    f\left( D_j \right)}}{3}\right) u,u}\right\rangle}
    \label{eq2.16}\\
    &\ge   \frac{2}{3}\left[ {f\left( {\left\langle{ \sum\limits_{j =
                    1}^n\Phi_j\left(\frac{{A_j + D_j}}{2}\right)u,u} \right\rangle }
        \right) +    f\left( {\left\langle { \sum\limits_{j =
                    1}^n\Phi_j\left(\frac{{B_j +D_j}}{2}\right)u,u} \right\rangle }
        \right)  +  f\left( {\left\langle{ \sum\limits_{j =
                    1}^n\Phi_j\left(\frac{{A_j + B_j}}{2}\right)u,u} \right\rangle }
        \right)} \right]
    \nonumber\\
    &\qquad+ \frac{1}{3} \left[\left\langle { \sum\limits_{j =
            1}^n\Phi_j \left( {f\left( {\left| {A_j -
                    \left\langle { \sum\limits_{j = 1}^n\Phi_j\left(\frac{{B_j + D_j}}{2}\right)u,u}
                    \right\rangle 1_{\mathcal{H}} } \right|} \right)} \right)u,u}
    \right\rangle +  f\left( {\left| {\left\langle { \sum\limits_{j =
                    1}^n\Phi_j
                \left( {\frac{{2A_j - B_j - D_j}}{6}} \right)u,u} \right\rangle }
        \right|} \right)\right.
    \nonumber \\
    &\qquad + \left\langle { \sum\limits_{j = 1}^n\Phi_j \left(
        {f\left( {\left| {D_j -
                    \left\langle { \sum\limits_{j = 1}^n\Phi_j\left(\frac{{A_j + B_j}}{2}\right)u,u}
                    \right\rangle 1_{\mathcal{H}} } \right|} \right)} \right)u_j,u_j} \right\rangle+   f\left( {\left| {\left\langle
            { \sum\limits_{j = 1}^n\Phi_j \left( {\frac{{2D_j - A_j- B_j}}{6}} \right)u,u} \right\rangle
        } \right|} \right)
    \nonumber\\
    &\qquad + \left. \left\langle {\sum\limits_{j = 1}^n\Phi_j \left(
        {f\left( {\left| {B_j -
                    \left\langle { \sum\limits_{j = 1}^n\Phi_j\left(\frac{{A_j + D_j}}{2}\right)u,u}
                    \right\rangle 1_{\mathcal{H}} } \right|} \right)} \right)u_j,u_j}
    \right\rangle  + f\left( {\left| {\left\langle
            { \sum\limits_{j = 1}^n\Phi_j \left( {\frac{{2B_j - A_j - D_j}}{6}} \right)u,u} \right\rangle
        } \right|} \right)\right].\nonumber
    \end{align}
\end{corollary}

\begin{proof}
    Let $E$ stands for  $A,B,D$. Since $ E\in
    \mathcal{B}^+\left(\mathcal{H}\right)$, then there exists
    $E_1,\cdots, E_n \in   \mathcal{B}^+\left(\mathcal{H}\right)$
    (where $E_j$ stands for $A_j,B_j, D_j$ for all $j=1,\cdots, n$)
    such that
    $E=  E_1  \oplus  \cdots  \oplus E_n \in \mathcal{B}^+\left(\mathcal{H} \oplus  \cdots  \oplus \mathcal{H} \right)$ for every unit vector $u=\left(u_1,\cdots,u_n\right)\in \mathcal{H} \oplus  \cdots  \oplus \mathcal{H}$. Let  $\Phi:\mathcal{B}^+\left(\mathcal{H} \oplus  \cdots  \oplus \mathcal{H} \right)\to \mathcal{B}^+\left(\mathcal{K}  \right)$
    be a positive normalized linear map  defined by $\Phi\left(E\right)=\sum\limits_{j = 1}^n {\Phi _j \left( {E_j } \right)}$.
    By utilizing  Theorem \ref{thm2.1} we get the desired result.
\end{proof}

\begin{corollary}
    \label{cor5}Let $A_j\in \mathcal{B}\left(\mathcal{H} \right)$ be   positive selfadjoint operators for each $j=1,\cdots,n$. Let $\Phi_j:\mathcal{B}\left(\mathcal{H} \right)\to \mathcal{B}\left(\mathcal{K} \right)$  be  positive linear maps such that $\sum\limits_{j = 1}^n {\Phi _j \left( {1_H } \right)}  = 1_K$.
    If $f:\left[0,\infty\right)\to \mathbb{R}$ is superquadratic, then we have
    \begin{multline}
    \left\langle {\sum\limits_{j = 1}^n\Phi_j \left( {f\left( A_j \right)} \right)u,u}
    \right\rangle
    \\
    \ge (\le)  f\left( {\left\langle {\sum\limits_{j = 1}^n\Phi_j \left( {A_j} \right)u,u} \right\rangle } \right)   +  \left\langle {\sum\limits_{j = 1}^n\Phi_j \left( {f\left( {\left| {A_j - \sum\limits_{j = 1}^n\left\langle {\Phi_j\left(A_j\right)u,u} \right\rangle 1_{\mathcal{H}} } \right|} \right)} \right)u,u} \right\rangle + f\left( {0} \right) \label{eq2.17}
    \end{multline}
    for each $x\in \mathcal{K}$ with $\|x\| =1$.
\end{corollary}

\begin{proof}
    Setting  $D_j=B_j=A_j$ for each $j=1,\cdots,n$, in \eqref{eq2.16}    we get the required result.
\end{proof}

 As a direct consequence of Theorem \ref{thm2.1}, the expected
operator version Popoviciu's inequality  for convex functions
would be as follows:
\begin{proposition}
    \label{prp1}Let $A,B,C\in \mathcal{B}\left(\mathcal{H} \right)$ be three positive selfadjoint operators, $\Phi:\mathcal{B}\left(\mathcal{H} \right)\to \mathcal{B}\left(\mathcal{K} \right)$  be a normalized positive linear map.   If $f:\left[0,\infty\right)\to \mathbb{R}$ is  non-negative and superquadratic, then $f$ is convex and
    \begin{align}
    &\left\langle {\Phi \left( {\frac{{f\left( A \right) + f\left( B \right) + f\left( D \right)}}{3}} \right)x,x} \right\rangle +f\left( {\left\langle {\Phi \left( {\frac{{A + B + D}}{3}} \right)x,x} \right\rangle } \right)
    \label{eq2.18}\\
    &\ge   \frac{2}{3}\left[f\left( {\left\langle {\Phi \left( {\frac{{A + B}}{2}} \right)x,x} \right\rangle } \right) + f\left( {\left\langle {\Phi \left( {\frac{{B+D}}{2}} \right)x,x} \right\rangle } \right)+ f\left( {\left\langle {\Phi \left( {\frac{{A + D}}{2}} \right)x,x} \right\rangle } \right)\right]\nonumber
    \end{align}
    for each $x\in \mathcal{K}$ with $\|x\| =1$.
\end{proposition}

\begin{proof}
    Since $f$ is non-negative superquadratic then by Lemma  \ref{lemma1}
    $f$  is convex and so that from \eqref{eq2.3}, we get
    \begin{align*}
    &\left\langle {\Phi \left( {\frac{{f\left( A \right) + f\left( B
                    \right) + f\left( D \right)}}{3}} \right)x,x} \right\rangle
    +f\left( {\left\langle {\Phi \left( {\frac{{A + B + D}}{3}}
            \right)x,x} \right\rangle } \right)
    \\
    &\ge  \frac{2}{3}\left[f\left( {\left\langle {\Phi \left(
            {\frac{{A + B}}{2}} \right)x,x} \right\rangle } \right) + f\left(
    {\left\langle {\Phi \left( {\frac{{B+D}}{2}} \right)x,x}
        \right\rangle } \right)+ f\left( {\left\langle {\Phi \left(
            {\frac{{A + D}}{2}} \right)x,x} \right\rangle }
    \right)\right]\nonumber
    \\
    &\qquad+ \frac{1}{3} \left[\left\langle {\Phi \left( {f\left(
            {\left| {A -
                    \left\langle {\Phi\left(\frac{{B + D}}{2}\right)x,x}
                    \right\rangle 1_{\mathcal{H}} } \right|} \right)} \right)x,x}
    \right\rangle +     f\left( {\left| {\left\langle
            {\Phi \left( {\frac{{2A - B - D}}{6}} \right)x,x} \right\rangle
        } \right|} \right)\right.
    \nonumber\\
    &\qquad\qquad+\left\langle {\Phi \left( {f\left( {\left| {D -
                    \left\langle {\Phi\left(\frac{{A + B}}{2}\right)x,x}
                    \right\rangle 1_{\mathcal{H}} } \right|} \right)} \right)x,x}
    \right\rangle+  f\left( {\left| {\left\langle {\Phi
                \left( {\frac{{2D - A - B}}{6}} \right)x,x} \right\rangle }
        \right|} \right)
    \nonumber\\
    &\qquad\qquad\left.  + \left\langle {\Phi \left( {f\left( {\left|
                {B -
                    \left\langle {\Phi\left(\frac{{A + D}}{2}\right)x,x}
                    \right\rangle 1_{\mathcal{H}} } \right|} \right)} \right)x,x}
    \right\rangle + f\left( {\left| {\left\langle
            {\Phi \left( {\frac{{2B - A - D}}{6}} \right)x,x} \right\rangle
        } \right|} \right)\right]\nonumber
    \\
    &\ge  \frac{2}{3}\left[f\left( {\left\langle {\Phi \left(
            {\frac{{A + B}}{2}} \right)x,x} \right\rangle } \right) + f\left(
    {\left\langle {\Phi \left( {\frac{{B+D}}{2}} \right)x,x}
        \right\rangle } \right)+ f\left( {\left\langle {\Phi \left(
            {\frac{{A + D}}{2}} \right)x,x} \right\rangle }
    \right)\right]\nonumber
    \end{align*}
    which gives \eqref{eq2.18}.
\end{proof}

\begin{proposition}
    \label{prp2}Let $A,B,D\in \mathcal{B}\left(\mathcal{H} \right)$ be three  selfadjoint operators, $\Phi:\mathcal{B}\left(\mathcal{H} \right)\to \mathcal{B}\left(\mathcal{K} \right)$  be a normalized positive linear map  and  $f:\left[0,\infty\right)\to \mathbb{R}$ be a differentiable    function with $f(0)=f^{\prime}(0)=0$. If $f^{\prime}$ is convex (concave),  then $f$ is super(sub)quadratic and
    \begin{align}
    &\left\langle {\Phi \left( {\frac{{f^{\prime}\left( A \right) + f^{\prime}\left(
                    B \right) + f^{\prime}\left( D \right)}}{3}} \right)x,x} \right\rangle
    +f^{\prime}\left( {\left\langle {\Phi \left( {\frac{{A + B + D}}{3}}
            \right)x,x} \right\rangle } \right)
    \label{eq2.19}\\
    &\ge (\le) \frac{2}{3}\left[f^{\prime}\left( {\left\langle {\Phi \left(
            {\frac{{A + B}}{2}} \right)x,x} \right\rangle } \right) + f^{\prime}\left(
    {\left\langle {\Phi \left( {\frac{{B+D}}{2}} \right)x,x}
        \right\rangle } \right)+ f^{\prime}\left( {\left\langle {\Phi \left(
            {\frac{{A + D}}{2}} \right)x,x} \right\rangle }
    \right)\right]\nonumber
    \end{align}
    for each $x\in \mathcal{K}$ with $\|x\| =1$.
\end{proposition}

\begin{proof}
    The superquadratic of $f$ follows from Lemma \ref{lemma2}. To
    obtain the inequality \eqref{eq2.19} we apply the same technique
    considered in the proof of Theorem \ref{thm2.1}, by  applying
    \eqref{eq1.3} for $f^{\prime}$ instead of  \eqref{eq1.4} for $f$,
    so that we get the required result.
\end{proof}

\begin{proposition}
    \label{prp3}Let $A,B,D\in \mathcal{B}\left(\mathcal{H} \right)$ be three  selfadjoint operators, $\Phi:\mathcal{B}\left(\mathcal{H} \right)\to \mathcal{B}\left(\mathcal{K} \right)$  be a normalized positive linear map and   $g:\left[0,\infty\right)\to \mathbb{R}$ be a  continuous function. If $g$ is convex (concave) and $g(0)=0$,   then
    \begin{align}
    &\left\langle {\Phi \left( {\frac{{g\left( A \right) + g\left(
                    B \right) + g\left( D \right)}}{3}} \right)x,x} \right\rangle
    +g\left( {\left\langle {\Phi \left( {\frac{{A + B + D}}{3}}
            \right)x,x} \right\rangle } \right)
    \label{eq2.20}\\
    &\ge (\le) \frac{2}{3}\left[g\left( {\left\langle {\Phi \left(
            {\frac{{A + B}}{2}} \right)x,x} \right\rangle } \right) + g\left(
    {\left\langle {\Phi \left( {\frac{{B+D}}{2}} \right)x,x}
        \right\rangle } \right)+ g\left( {\left\langle {\Phi \left(
            {\frac{{A + D}}{2}} \right)x,x} \right\rangle }
    \right)\right]\nonumber
    \end{align}
    for each $x\in \mathcal{K}$ with $\|x\| =1$.
\end{proposition}

\begin{proof}
    Applying Corollary \ref{cor5} for $G\left( t \right) = \int_0^t
    {g\left( s \right)ds}$, $t\in \left[0,\infty\right)$, then it's
    easy to observe that $G\left(0\right)=G^{\prime}\left(0\right)=0$
    and $G^{\prime}\left(t\right)=g\left(t\right)$ is convex (concave)
    for all $t\in \left[0,\infty\right)$.
\end{proof}

Inequality \eqref{eq2.20} holds with   more weaker conditions,
indeed   neither continuity assumption nor the image of $0$ is
needed, it is hold just with convexity assumption, as follows:

\begin{theorem}
    \label{thm3}Let $A,B,D\in \mathcal{B}\left(\mathcal{H} \right)$ be three  selfadjoint operators with
    $\spe\left(A\right),\spe\left(B\right),\spe\left(D\right)\subset
    \left[\gamma,\Gamma\right]$  for some real numbers $\gamma,\Gamma$
    with $\gamma<\Gamma$. Let $\Phi:\mathcal{B}\left(\mathcal{H} \right)\to \mathcal{B}\left(\mathcal{K} \right)$  be a normalized positive linear map. If $f:
    \left[\gamma,\Gamma\right]\to \mathbb{R}$  is convex
    (concave) function, then
    \eqref{eq2.20} holds   for each $x\in \mathcal{K}$ with $\|x\| =1$. The inequality is satisfied with $f(t)=t$.
\end{theorem}

\begin{proof}
    Applying the same technique considered in the proof of Theorem
    \ref{thm2.1}, by  applying \eqref{eq1.3} for $f$   instead of
    \eqref{eq1.4} for $f$.
\end{proof}

\begin{remark}
    Employing \eqref{eq2.20} for $g(x)=|x|$, $x\in \mathbb{R}$ then we
    observe  that
    \begin{align}
    &  \left| {\left\langle { \Phi \left( A+C \right) x ,x }
        \right\rangle } \right| + \left|\left\langle { \Phi \left( B+C
        \right)x ,x } \right\rangle \right| + \left|\left\langle { \Phi
        \left(A+B\right)x,x} \right\rangle \right|
    \nonumber\\
    &\le       \left|   {   \left\langle { \Phi \left( A+B+C \right) x
            ,x } \right\rangle }\right| +  \left\langle {\Phi \left( { \left|
            A \right| +  \left| B \right| +  \left| C \right|} \right)x ,x}
    \right\rangle, \label{eq2.21}
    \end{align}
    which gives  the operator version of Hlawka's inequality for
    positive linear maps of selfadjoint operators in Hilbert space.
    Furthermore, by taking the supremum in \eqref{eq2.21} over $x\in
    \mathcal{K}$ with $\|x\|=1$, we obtain the following Hlawka's norm
    inequality
    \begin{align*}
    &\left| {\left\| { \Phi \left( A+C \right)   } \right\|} \right| +
    \left|\left\| { \Phi \left( B+C \right) } \right\| \right| +
    \left|\left\langle { \Phi \left(A+B\right)  } \right\| \right|
    \\
    &\le       \left|   {   \left\| { \Phi \left( A+B+C \right)  }
        \right\| }\right| +  \left\| {\Phi \left( { \left| A \right| +
            \left| B \right| +  \left| C \right|} \right) } \right\|.
    \end{align*}
    Generally, the  Popoviciu's extension of Hlawka's  norm
    inequality can be presented in the form:
    \begin{align*}
    &\frac{2}{3}\left[g\left(\left\| { \Phi \left( \frac{A+C}{2}
        \right)   } \right\| \right) +g\left(\left\|{ \Phi \left(
        \frac{B+C}{2} \right)  } \right\| \right) +g\left(\left\| { \Phi
        \left( \frac{A+B}{2} \right) } \right\| \right) \right]
    \nonumber\\
    &\le      g\left(   {   \left\| { \Phi \left( \frac{A+B+C}{3}
            \right)   } \right\| }\right) +  \left\| {\Phi \left(
        {\frac{{g\left( A \right) + g\left( B \right) + g\left( C
                    \right)}}{3}} \right)  } \right\|
    \end{align*}
    for every positive  linear map $\Phi$ and convex increasing function $g$.
\end{remark}

\begin{corollary}
    Let $A_j,B_j,D_j\in \mathcal{B}\left(\mathcal{H} \right)$ be three   selfadjoint operators with
    $\spe\left(A_j\right),\spe\left(B_j\right),\spe\left(D_j\right)\subset
    \left[\gamma,\Gamma\right]$  for some real numbers $\gamma,\Gamma$
    with $\gamma<\Gamma$ and for every $j=1,\cdots,n$. Let $\Phi_j:\mathcal{B}\left(\mathcal{H} \right)\to \mathcal{B}\left(\mathcal{K} \right)$  be a positive   linear map.  such that $\sum\limits_{j = 1}^n {\Phi _j \left( {1_H } \right)}  = 1_K$. If $ f:
    \left[\gamma,\Gamma\right]\to \mathbb{R}$   is convex
    (concave) function, then
    \begin{multline*}
    {f\left(\left\langle {\sum\limits_{j = 1}^n\Phi_j\left(\frac{A_j+B_j+D_j}{3}\right)
            u,u} \right\rangle\right)}+
    {\left\langle { \sum\limits_{j = 1}^n\Phi_j\left(\frac{{f\left( A_j \right) +f\left( B_j \right) +
                    f\left( D_j \right)}}{3}\right) u,u}\right\rangle}
    \\
    \ge (\le)  \frac{2}{3}\left[ {f\left(
        {\left\langle{ \sum\limits_{j = 1}^n\Phi_j\left(\frac{{A_j + D_j}}{2}\right)u,u} \right\rangle }
        \right) +    f\left( {\left\langle { \sum\limits_{j = 1}^n\Phi_j\left(\frac{{B_j +D_j}}{2}\right)u,u}
            \right\rangle } \right)  +  f\left( {\left\langle{ \sum\limits_{j = 1}^n\Phi_j\left(\frac{{A_j +
                        B_j}}{2}\right)u,u} \right\rangle } \right)} \right]
    \end{multline*}
    for each $u\in \mathcal{K}$ with $\|u\| =1$.
\end{corollary}


\begin{thebibliography}{5}
    \setlength{\itemsep}{3pt}

    \bibitem{S} S. Abramovich, On superquadraticity, {\em J. Math. Inequal. }, {\bf 3} (3) (2009), 329--339.

    \bibitem{SIP} S. Abramovich, S. Iveli\'{c} and J. Pe\v{c}ari\'{c}, Improvement of Jensen-Steffensen's
    inequality for superquadratic functions, {\em Banach J. Math. Anal.},  {\bf4} (1) (2010), 146-158.


    \bibitem{SJS}  S. Abramovich, G. Jameson and  G. Sinnamon,  Refining Jensen's inequality, {\em Bull.
        Math. Soc. Sci. Math. Roumanie}, {\bf 47} (2004), 3--14.

    \bibitem{AD} R.P. Agarwal and S.S. Dragomir, A survey of Jensen type inequalities for functions
    of selfadjoint operators in Hilbert spaces, {\em  Comput. Math. Appl.}, {\bf  59} (2010), 3785--3812.

    \bibitem{BMP} J. Bari\'{c}, A. Matkovi\'{c} and J. Pe\v{c}ari\'{c},  A variant of the Jensen--Mercer operator
    inequality for superquadratic functions, {\em Math. Comput. Modelling},  {\bf51} (2010) 1230--1239.

    \bibitem{BPV}S.  Bani\'{c}, J. Pe\v{c}ari\'{c} and S. Varo\v{s}anec,  Superquadratic functions and refinements of some classical inequalities, {\em J. Korean Math. Soc.}, {\bf 45}, (2) (2008), 513--525.
    
    \bibitem{BNP}M. Bencze, C.P. Niculescu and F. Popovici, Popoviciu's inequality for functions of several
    variables, {\em J. Math. Anal. Appl.,}  {\bf 365} (2010),
    399--409.
    
\bibitem{CHP}    P. Chansangiam P. Hemchote and P. Pantaragphong,    Generalizations of Bohr inequality for Hilbert space operators, {\em J. Math. Anal. Appl.}, {\bf356} 
(2009),  525--536.

 \bibitem{CP}W.--S. Cheung and J. Pe\v{c}ari\'{c}, Bohr's inequalities for Hilbert space operators,  {\em J. Math. Anal. Appl.}, {\bf323} 
 (2006),  403--412.
 

 
    \bibitem{SD1}S.S. Dragomir,  Operator inequalities
    of the Jensen, \v{C}eby\v{s}ev and Gr\"{u}ss type, Springer, New
    York,  2012.
    
 \bibitem{FZ}M.Fuji and H. ZUO, Matrix order in Bohr inequality for operators,  {\em  Banach J. Math. Anal.}, {\bf4}  (2010),  21--27.

    \bibitem{FMPS} T. Furuta, J. Mi\'{c}i\'{c}, J. Pe\v{c}ari\'{c} and Y. Seo,  Mond-Pe\v{c}ari\'{c} method in operator inequalities. Inequalities for bounded    selfadjoint operators on a Hilbert space, Element, Zagreb, 2005.

\bibitem{G}D. Grinberg, Generalizations of Popoviciu's inequality, (2008),
arXiv:0803.2958v1.

 \bibitem{H} O.  Hirzallah, Non-commutative operator Bohr inequality,  {\em J. Math. Anal. Appl.}, {\bf282} 
 (2003),  578--583.
 
 
    \bibitem{J}J. Jensen, Sur les fonctions convexes et les in\'{e}galit\'{e}s
    entre les valeurs moyennes, {\em Acta Math.}, {\bf30} (1906),
    175--193.
    
    \bibitem{KLPP}  M. Krni\'{c}, N. Lovrin\v{c}evi\'{c}, J. Pe\v{c}ari\'{c}, J. Peri\'{c}, Superadditivity and monotonicity of the Jensen-type functionals: New Methods for improving the Jensen-type Inequalities in Real and in Operator Cases, Element, Zagreb, 2016.
    
    

    \bibitem{K}M. Kian, Operator Jensen inequality for superquadratic functions, {\em Linear Algebra and its Applications},
    {\bf 456},  (2014),  82--87.

    \bibitem{KS} M. Kian and S.S. Dragomir, Inequalities involving superquadratic
    functions and operators, {\em Mediterr. J. Math.}, {\bf11} (4)
    (2014),  1205--1214.



\bibitem{MPP} J. Mi\'{c}i\'{c}, J. Pe\v{c}ari\'{c} and J.  Peri\'{c},  Extension of the refined Jensen’s operator inequality
with condition on spectra, {\em Ann. Funct. Anal.}, {\bf3} (1) (2012), 67--85.



    \bibitem{MPF}D.S. Mitrinovi\'{c}, J. Pe\v{c}ari\'{c} and A.M. Fink, Classical and New
    Inequalities in Analysis, Kluwer Academic, Dordrecht, 1993.

    
    
    \bibitem{MM}F.C. Mitroi and N. Minculete, On the Jensen functional and superquadraticity, {\em Aequat. Math.}, {\bf90} (4) (2016), 705--718.
    
    

    \bibitem{MP}  B. Mond  and Pe\v{c}ari\'{c}, Convex inequalities in Hilbert space, {\em Houston J.
        Math.}, {\bf19} (1993), 405--420.
    
    \bibitem{MR}M.S. Moslehian and R. Raj\'{c} Generalizations of Bohr's inequality in Hilbert $C^*$-modules, {\em Linear and Multilinear Algebra}, {\bf58}   (2010),  323--331.

    

    \bibitem{N}C.P. Niculescu, Convexity according to the geometric mean,
    \textit{Math. Inequal. Appl.}, {\bf3} (2) (2000), 155--167.

    \bibitem{NF}C.P. Niculescu, F. Popoviciu,  A refinement of Popoviciu's
    inequality, {\em Bull. Math. Soc. Sci. Math. Roumanie}, Tome {\bf
        49} (3)  (2006), 285--290.

    \bibitem{NP}C.P. Niculescu, L.E. Persson, Convex Functions and Their
    Applications. A Contemporary Approach, CMS Books Math., vol. 23,
    Springer-Verlag, New York, 2006.
    
    




    \bibitem{P}T. Popoviciu, Sur certaines in\'{e}galit\'{e}s qui
    caract\'{e}risent les fonctions convexes, Analele Stiintifice
    Univ.  Al. I. Cuza, Iasi, Sectia Mat., 11 (1965), 155--164.





     


\bibitem{SS} D.M. Smiley and M.F. Smiley, The polygonal inequalities, {\em Amer. Math.
	Mon.}, {\bf 71} (7) (1964), 755--760.

\bibitem{TTH}S.-E. Takahasi, Y. Takahashi and A. Honda, A new interpretation of
Djokovi\'{c}'s inequality, {\em  J. Nonlinear Convex Anal.}, {\bf
	1} (3) (2000), 343--350.

\bibitem{TTMT}S.-E. Takahasi, Y. Takahashi, S. Miyajima and H. Takagi,  Convex
sets and inequalities, {\em J. Inequal. Appl.}, {\bf 2} (205),
107--117.

\bibitem{VS}P.M. Vasi\'{c} and L.R. Stankovi\'{c}, Some inequalities for
convex functions, {\em Math. Balkanica}, {\bf6} (1976), 281--288.



    \bibitem{W}R. Whitty, A generalised Hlawka inequality  by D. Smiley \& M.  Smiley,  Theorem
    of The day.





    \bibitem{Z} F. Zhang, On the Bohr inequality of operators,  {\em J. Math. Anal. Appl.}, {\bf333} 
 (2007),  1264--1271.





























 
\end{thebibliography}
\end{document}